\documentclass[11]{amsart}
\usepackage{amscd,amsmath,latexsym}
\usepackage[mathcal]{euscript}
\usepackage[all]{xy}
\CompileMatrices

\newtheorem{thm}{Theorem}[section]
\newtheorem{cor}[thm]{Corollary}
\newtheorem{lem}[thm]{Lemma}
\newtheorem{prop}[thm]{Proposition}
\theoremstyle{definition}

\newtheorem{defin}[thm]{Definition}
\theoremstyle{claim}

\theoremstyle{definition}
\newtheorem{exm}[thm]{Example}

\newtheorem{conjecture}[thm]{Conjecture}
\theoremstyle{remark}
\newtheorem*{rem}{Remark}

\DeclareMathOperator{\hocolim}{hocolim}

\newcommand{\R}{{\mathbb R}}

\newcommand{\Z}{{\mathbb Z}}

\newcommand{\F}{{\mathbb F}}

\usepackage[pdftex]{hyperref}

\begin{document}

\title{Higher commutativity and nilpotency in finite groups.}

\author[E.~Torres Giese]{E.~Torres-Giese}
\address{Department of Mathematics,
University of Michigan, Ann Arbor MI 48109, USA}
\email{etorresg@umich.edu}

\date{\today}
\begin{abstract}
We consider ordered tuples in finite groups generating nilpotent subgroups. Given an integer $q$ we consider the poset of nilpotent subgroups of class 
less than $q$ and its corresponding coset poset. These posets give rise to a family of finite Dirichlet series parametrized by the nilpotency class of the subgroups, which in turn reflect probabilistic and topological invariants determined by these subgroups. Connections of these series to filtrations of 
classifying spaces of a group are discussed.   
\end{abstract}

\maketitle

\section{Introduction}

In 1936 P.~Hall addressed to some extent the problem of computing the probability that a randomly chosen
ordered $s$-tuple of elements in a finite group $G$ generates the group. Hall proved that this probability, denoted by
 $P(G,s)$, can be expressed as a finite Dirichlet series. For instance (see~\cite{Boston}),

\[ P(PSL(2,7),s)=1- \frac{14}{7^s}-\frac{8}{8^s}+\frac{21}{21^s}+\frac{28}{28^s}+\frac{56}{56^s}-\frac{84}{84^s} .\]

Hall's results allow us to think of this as a function over the complex numbers. In~\cite{Brown} this function was related
to the coset poset of a finite group. This poset is the set of all proper cosets of a group ordered by inclusion. If we denote
this poset by $\mathcal{C}(G)$, then Brown (following Bouc) proved that 
\[ P(G,-1)=1- \chi(\mathcal{C}(G)) .\]

The reciprocal of $P(G,s)$ is called the probabilistic zeta function of $G$ (see~\cite{Boston}).

\medskip
We can twist the probability question to ask the following: given a finite group $G$, what is the probability that a randomly chosen
ordered $s$-tuple generates an abelian subgroup of $G$? If we denote this probability by $P_2(G,s)$, then we have
\[ P_2(G,s)=\frac{|Hom(\Z^s,G)| }{|G|^s} \]
since the set of ordered commuting $s$-tuples in $G$ can be identified with the set of group homomorphisms $Hom(\Z^s,G)$.
The number $P_2(G,s)$ was also studied in~\cite{Lescot} under the name of multiple commutativity degree.
The set (space if $G$ has some topology) $Hom(\Z^s,G)$ appears in different contexts of mathematics such as 
Differential Geometry, Group Cohomology, and $K$-theory (see~\cite{AC}). The sets $Hom(\Z^s,G)$, $s\geq 0$, can
be assembled to form a simplicial space. The realization of this simplicial space is denoted by $B(2,G)$ and it is the first layer 
of a filtration of the classifying space $BG$. 
The space $B(2,G)$ turns out to capture delicate information about the group $G$. 
For instance, when $G$ is finite of odd order the pull-back
fibration $E(2,G)\to B(2,G)$, obtained from the universal $G$-bundle over $BG$, captures the celebrated
Feit-Thompson Theorem;  and from the point of view of group cohomology, it can be seen that there is an 
$F$-monomorphism $ H^*(BG;\F_p)\to H^*(B(2,G);\F_p) $.  In this paper we will show that the space $B(2,G)$ is closely related
to the probability function $P_2(G,s)$. 

Let $\mathcal{N}_2(G)$ be the poset of all abelian subgroups of $G$. It was proved in~\cite{ACT} that the space $B(2,G)$ is a 
homotopy colimit, namely
\[ B(2,G)\simeq \underset{A\in \mathcal{N}_2(G)}\hocolim~ BA ,\]  
and this can be used to show that 
\[ E(2,G)\simeq \underset{A\in \mathcal{N}_2(G)}\hocolim~ G/A .\]  

For instance, if $q=2^n$ and $n\geq 2$, then
\[ B(2,SL(2,\F_q)) \simeq [\bigvee^{q+1} B(\mathbb Z/q)^n ]
\vee [\bigvee^{\frac{1}{2}q(q-1)}B\mathbb Z/(q+1)] \vee [\bigvee^{\frac{1}{2}q(q+1)} B\mathbb Z/(q-1)],\]
and
\[ E(2,SL(2,\F_q))=\bigvee_{(q^2-1)^2(q+1)-q^2(q^2+1)+1} S^1 .\]

Lets define $\mathcal{C}_2(G)$ as the poset of all cosets of proper abelian subgroups. 
In this context we have the following result

\begin{thm} If $G$ is a finite group, then 
\begin{enumerate}
\item there is a homotopy equivalence $\mathcal{C}_2(G)\simeq E(2,G)$,
\item $P_2(G,s)$ is a finite Dirichlet series, and 
\item $P_2(G,-1)=\chi(E(2,G))$. 
\end{enumerate}
\end{thm}

For instance,
\[ P_2(A_5,s)=\frac{6}{12^s}+\frac{5}{15^s}+\frac{10}{20^s}-\frac{20}{60^s}. \]

\bigskip
The poset $\mathcal{C}_2(G)$ has some intriguing properties. Brown in his paper~\cite{Brown} asked: Can we characterize finite 
solvable groups in terms of the combinatorial topology of the coset poset?
The answer is yes for groups of odd order if instead we consider the poset $\mathcal{C}_2(G)$ since, as noted in~\cite{ACT}, the
following statements are equivalent 
\begin{enumerate}
\item Every group of odd order is solvable.
\item The  map $H_1( \mathcal{C}_2(G))\to H_1(B(2,G))$ is not onto for every group $G$ of odd order.  
\end{enumerate}

\medskip
This probabilistic setting can be extended to higher classes of nilpotency. We will show that there is a finite Dirichlet series whose value at a positive integer $s$ is equal to
\[ \frac{ |\{ (g_1,\ldots, g_s)\in G^s: \Gamma^q(\langle g_1,\ldots, g_s\rangle\}=1 |}{|G|^s} \]
where $\Gamma^q$ stands for the $q^{th}$-stage of the lower central series. Note that this latter ratio is precisely the probability
that a randomly chosen $s$-tuple generates a nilpotent subgroup of class less than $q$. 
The paper is organized as follows: in Section 2 we cover some background; in Section 3 we relate  
the coset posets to probabilistic invariants via M\"obius functions; in Section 4 we explore some properties of the finite Dirichlet
series afforded by nilpotent subgroups; in Section 5 we provide an explicit formula to compute the number of commuting elements in 
a symmetric group; and in Section 6 we present upper and lower bounds for the probabilistic invariants aforementioned. 
In this paper a group will always be finite unless otherwise stated.


\section{Preliminaries}

Recall that the lower central series of a group $K$ is defined inductively by $\Gamma^1(K)=K$ and $\Gamma^q(K)=[\Gamma^{q-1}(K),K]$.
We say that a group $K$ is nilpotent of class $c$ if $\Gamma^{c+1}(K)=1$ and $\Gamma^q(K)\neq 1$ for $q\leq c$. By convention we will write
$\Gamma^\infty(K)=1$. Let $G$ be a group, and fix $q\geq 2$. We define 
\[ B_n(G,q)=Hom(F_n/\Gamma^q_n,G)\]
and 
\[ E_n(G,q)=G\times Hom(F_n/\Gamma_n^q,G),\]
where $F_n$ is the free group on $n$ generators and $\Gamma^q_n=\Gamma^q(F_n)$. Note that $B_n(G,q)\subseteq G^n$ and
$E_n(G,q)\subseteq G^{n+1}$. The following functions provide the sets $B_*(G,q)$ and $E_*(G,q)$ with a simplicial structure: 
$d_i:E_n(q,G)\to E_{n-1}(q,G)$ for $0\le i\le n$,
and $s_j:E_n(q,G)\to E_{n+1}(q,G)$ for $0\leq j\leq n$, are given by
\[d_i(g_0,\ldots,g_n)=\left\lbrace\begin{array}{lr}
(g_0,\ldots,g_i\cdot g_{i+1},\ldots,g_n) & 0\leq i<n, \\
(g_0,\ldots,g_{n-1}) & i=n,\\
 \end{array}\right. \]
and
\[ s_j(g_0,\ldots,g_n)=(g_0,\ldots,g_j,e,g_{j+1},\ldots,g_n)\]
for $0\le j\le n$. Similarly, we have maps $d_i$ and $s_j$ for $B_*(q,G)$ defined in the same way, except that the
first coordinate $g_0$ is omitted and the map $d_0$ takes the form $d_0(g_1,\ldots,g_n)=(g_2,\ldots,g_n)$.
Note that $G$ acts on $E_*(G,q)$ by multiplication on the first coordinate $g(g_0,g_1\ldots,g_n)=(gg_0,g_1,\ldots,g_n)$, 
with orbit space homeomorphic to $B_*(q,G)$. We denote by $B(q,G)$ and $E(q,G)$ the geometric realizations of the aforementioned simplicial sets.
Thus we have $G$-bundles $E(q,G)\to B(q,G)$ for each $q\geq 2$, and fibrations $E(q,G)\to B(q,G)\to BG$.  
Unlike the classical situation, the space $E(q,G)$ is not necessarily contractible. Note that if $E(q,G)$ is contractible, then $B(q,G)\simeq BG$.
We can also consider the simplicial set of $BN_*(G)$ whose set of $n$-simplices is defined as   
 \[ \{ (g_1,\ldots,g_n)|\langle g_1,\ldots,g_n\rangle \neq G\}.\]
Likewise, we can define $NE_*(G)$ as the simplicial set whose $n$-simplices is defined as 
$G\times BN_n(G)$. The simplicial structure of these latter two spaces is the same as that of $B(q,G)$ and $E(q,G)$ respectively.
We will denote their realization by $BN(G)$ and $EN(G)$. Note that there is a fibration $EN(G)\to BN(G)\to BG$.

The inclusions $F_n/\Gamma^{q+1}_n \subseteq F_n/\Gamma^q_n$ induce the following filtrations:
\[ E(2,G)\subseteq \cdots \subseteq E(q,G)\subseteq E(q+1,G)\subseteq\cdots \subseteq E(\infty,G)=EG ,\]
and
\[ B(2,G)\subseteq \cdots \subseteq B(q,G)\subseteq B(q+1,G)\subseteq\cdots \subseteq B(\infty,G)=BG. \]
Moreover, $B(q,G)\subseteq BN(G)\subset BG$, and $E(q,G)\subseteq EN(G)\subset EG$ for all $q \geq 2$.

\begin{defin} Let $q\ge 2$. We define $\mathcal{N}_q(G)$ as the poset of all nilpotent, proper subgroups of $G$ of class $< q$, ordered by
inclusion. When $q=\infty$ this poset becomes the poset of all subgroups of $G$, 
which we will denote by $\mathcal{L}(G)$.
\end{defin}

Consider the functor $F_q:\mathcal{N}_q(G)\to G$--$Sets$ given by $H\mapsto G/H$. In~\cite{ACT} it was proven that
if $\Gamma^q(G)\neq 1$, then 
$B(q,G)\simeq \hocolim (F_q)_{hG}$, that is,
\[ B(q,G)\simeq \underset{H\in\mathcal{N}_q(G)}\hocolim BH. \]
Likewise we have
\[ NB(G)\simeq \underset{H\in\mathcal{L}(G) -\{ G\}}\hocolim BH \]
The isomorphism  $(\hocolim F_q)_{hG} \cong \hocolim (F_q)_{hG}$ (see~\cite{Dwyer}) shows that 
\[ E(q,G)\simeq \underset{H\in\mathcal{N}_q(G)}\hocolim G/H ,\]
and
\[ EN(G)\simeq \underset{H\in\mathcal{L}(G) -\{ G\}}\hocolim G/H .\]

Suppose that $H$ and $K$ are subgroups of $G$, so one can see that $xH\subseteq yK$ if and only if, $H\subseteq K$ and $xK=yK$. 
This allows us to identify the images of the arrows under the functor $F_q$ with those of the poset of all proper cosets of $G$.
This latter poset is denoted by $\mathcal{C}(G)$.

\begin{defin} Let $q\geq 2$. We define $\mathcal{C}_q(G)$  as the poset of proper cosets of subgroups of $G$ of nilpotency class
less than $q$. That is, $\mathcal{C}_q(G)=\{ xH\in\mathcal{C}(G) |  \Gamma^{q}(H)=1\}$.
When $q=\infty$ this poset becomes the poset of all proper cosets of $G$, that is $\mathcal{C}_\infty(G)=\mathcal{C}(G)$.
\end{defin}

Thus, for finite $q$, we have the equivalences: $E(q,G)\simeq \mathcal{C}_q(G)$ if $\Gamma^q(G)\neq 1$, and  $EN(G)\simeq \mathcal{C}(G)$. 
If $\Gamma^q(G)=1$, then $E(q,G)\simeq EG$. But notice that $EG=E(\infty,G)$ is not, in general, homotopy equivalent to $\mathcal{C}(G)$.
Note that $B(q,G)$ is always connected since $B_0(q,G)=\{ 1\}$. This implies that $E(q,G)$ is always connected. On the other hand, 
it was proven in~\cite{Brown}  that $EN(G)$ is connected if and only if $G$ is not isomorphic to cyclic group of order a prime power. So, 
$G$ is not isomorphic to  a cyclic group of prime power order if and only if, $BN(G)$ is connected.


\section{The coset poset}

We can make our identification of $E(q,G)$ more efficient by considering only the maximal subgroups of $G$ of nilpotency class $<q$.

\begin{defin}
Let $\mathcal{M}_q(G)$ be the poset determined by the maximal nilpotent subgroups of $G$ of class $<q$; that is, any subgroup in
$\mathcal{M}_q(G)$ is an intersection of maximal nilpotent subgroups of $G$ of class $<q$. The corresponding coset poset is
defined as $\mathcal{MC}_q(G)=\{ xH\in\mathcal{C}(G)\ |\ H\in\mathcal{M}_q(G)\}$
\end{defin}

Recall the following Lemma.

\begin{lem}[Quillen, Webb-Thevenaz] Let $f:P\to Q$ be a $G$-map of $G$-posets. If the subposet $f_{\geq y} =\{ x\in P| f(x)\geq y\}$ is
$G_y$-contractible for all $y\in Q$, then $f$ is a homotopy equivalence.
\end{lem}
\begin{rem} The group $G$ acts on the coset poset in two ways: (1) by translation, and (2) by conjugation $g(xH)g^{-1}= gxg^{-1}(gHg^{-1})$.
We also have the action of  $Hol(G)=G \times_\tau Aut(G)$ given by $(g\sigma)\cdot xH = g\sigma(x)\sigma(H)$.
\end{rem}

Note that if $f:\mathcal{MC}_q(G)\to \mathcal{C}_q(G)$ denotes the inclusion map, 
then for $xH\in \mathcal{C}_q(G)$ the poset $f_{\geq xH}$ has an initial element, namely the coset $xM_H$, where $M_H$ is the intersection 
of all the subgroups in $\mathcal{M}_q(G)$ containing $H$. Moreover, one can see that $xM_H$ is fixed by $G_{xH}$ (under both the translation
and conjugation action). 

We also have ($G$-equivariant) inclusions $\mathcal{C}_2(G)\subseteq \mathcal{C}_3(G)\subseteq \cdots \subseteq \mathcal{C}_q(G)\subseteq \cdots \subseteq \mathcal{C}(G)$.

\begin{prop} For any finite group $G$ such that $\Gamma^q(G)\neq 1$, there are homotopy equivalences 
\[ E(q,G)\simeq \mathcal{MC}_q(G) \]
and
\[ B(q,G) \simeq \underset{A\in\mathcal{M}_q(G)}\hocolim\ G/A  .\]
\end{prop}

\medskip
We will now focus on computing the Euler characteristic of $E(q,G)$. For this purpose we will use M\"obius functions. 

\begin{defin} If $\mathcal{X}$ is a finite poset, then the M\"obius funtion of
this poset is a function $\mu_\mathcal{X}:\mathcal{X}\times\mathcal{X} \to \Z$ determined recursively by the following equations
\[ \mu_\mathcal{X}(a,a)=1, \]
and for $a<b$
\[ \mu_\mathcal{X}(a,b) = - \sum_{a\leq u< b} \mu_\mathcal{X}(a,u) .\]
\end{defin}

Let $\mathcal{X}$ be a subposet of the poset of all subgroups of $G$, and assume that $G$ is in $\mathcal{X}$. Then Hall showed that
$\mu_\mathcal{X}(H,G)$ is computed by a signed sum of the number of chains of subgroups in $\mathcal{X}$ from $H$ to $G$. 
We say that $H=K_0<K_1<\cdots<K_n=G$ has length $n$ and it is counted with the sign $(-1)^n$. Now, (following Bouc) note that a chain
of the form $H=K_0<K_1<\cdots<K_n=G$ corresponds to $|G:H|$ chains of cosets $C_0<C_1<\cdots <C_{n-1}<G$, where $C_0=xH$ and
$x$ is a representative of the $|G:H|$ cosets of $H$ in $G$.  
As the chains $C_0<C_1<\cdots <C_{n-1}$ form the simplices of the simplicial complex afforded by $\mathcal{C}_\mathcal{X}$, we see that
\[ 1-\chi(\mathcal{C}_\mathcal{X})=  \sum_{H\leq G} \mu_\mathcal{X}(H,G)|G:H|, \]
and thus 
\[ \chi(\mathcal{C}_\mathcal{X})= - \sum_{H\in\mathcal{X}-\{ G\} } \mu_\mathcal{X}(H,G)|G:H|. \]
 
\begin{defin} Let $q\geq 2$. We define $\mathcal{L}_q(G)$ as the poset $\mathcal{N}_q(G)\cup \{ G\}$, and $\mu_q$ as the
M\"obius function of the poset $\mathcal{L}_q(G)$.
\end{defin} 

\begin{rem} From now on we will assume that $\Gamma^q(G)\neq 1$ (in particular $q$ has to be finite).
\end{rem}

\begin{prop} Let $G$ be a finite group. Then
\[\chi(E(q,G)) = 1- \sum_{H\in\mathcal{L}_q(G)} \mu_q(H,G) |G:H| = -\sum_{H\in\mathcal{M}_q(G)} \mu_q(H,G) |G:H| .\]
\end{prop}

\begin{rem} In the formula above the sum can run either over $\mathcal{M}_q(G)$ or $\mathcal{N}_q(G)$. There is a topological reason
for this as we have seen, but there is a combinatorial one which follows from the fact that if $\mathcal{X}$ is a suposet of the poset of all subgroups of 
$G$ (assume $G\in \mathcal{X})$, then $\mu_\mathcal{X}(H,G)=0$ if $H$ is not an intersection of maximal subgroups in $\mathcal{X}$ 
(see Theorem 2.3 in~\cite{Hall}).
Another important property that facilitates the computation of the M\"obius function is that if $\mathcal{X}$ is invariant under $Aut(G)$, then so is 
$\mu_\mathcal{X}$. 
\end{rem}

\begin{exm} A finite group $G$ is said to be TC-group if $[g,h]=1=[h,k]$ implies $[g,k]=1$ for all non-central elements. The poset of maximal abelian
subgroups of a TC-group $G$ consists of the maximal abelian subgroups and the center of $G$ which is the intersection of any pair of distinct maximal
abelian subgroups of $G$. Then we obtain the formula
\[ 1- \chi(E(2,G)) = 1 - |G:Z(G)| +\sum_{1\leq i\leq N}(|G:Z(G)| -|G:M_i|) \]
where $M_1,\ldots,M_N$ are the maximal abelian subgroups of $G$ (see~\cite{ACT}). The family of TC-groups includes groups such as 
any non-abelian group of order less than 24, dihedral groups, quaternion groups, and $SL(2,\F_{2^n})$ with $n\geq 2$. See~\cite{Sc}, p.~519, for
a classification of these groups. When $G$ is a TC-group it turns out that $E(2,G)$ has the homotopy type of a wedge of $1-\chi(E(2,G))$ circles 
(see~\cite{ACT}).
\end{exm}

\section{Probability}

P.~Hall defined $\phi(G,s)$ as the number of ordered $s$-tuples $(g_1,\ldots,g_s)\in G^s$ such that $\langle g_1,\ldots, g_s\rangle=G$. 
Thus the probability that a randomly chosen $s$-tuple generates $G$ is given by 
\[ P(G,s)=\frac{\phi(G,s)}{|G|^s} \]

Analogously we can define $P_q(G,s)$ as the probability that a randomly chosen ordered $s$-tuple generates a nilpotent subgroup of 
class less than $q$. Of course we have
\[ P_q(G,s)=\frac{ |Hom(F_s/\Gamma^q,G)|}{|G|^s} \]
since $Hom(F_s/\Gamma^q,G)=\{ (g_1,\ldots,g_s)\in G^s| \Gamma^q(\langle g_1,\ldots,g_s\rangle )=1\}$.

\begin{lem} Let $\mu$ be the M\"obius function of a poset $\mathcal{X}$, and define the function $\tilde{\mu}$ inductively by 
\[ \tilde\mu (a,a)=1 \]
\[ \tilde\mu (a,b) = - \sum_{a< u\leq b} \mu(u,b) .\]
Then $\mu=\tilde\mu$.
\end{lem}

\begin{prop}\label{P-as-Dirichlet} Let $\mu_q$ be the M\"obius function of the poset $\mathcal{L}_q(G)=\mathcal{N}_q(G)\cup\{ G\}$. Then
\[ P_q(G,s)=-\sum_{A\in \mathcal{N}_q(G)} \frac{ \mu_q(A,G)}{|G:A|^s} \]
\end{prop}
\begin{proof}
Note that 
\[ \sum_{A\in \mathcal{N}_q(G)} \mu_q(A,G) |A|^s = \sum_{A\in \mathcal{N}_q(G)} \sum_{g_1,\ldots,g_s \in A} \mu_q(A,G) .\]
If we interchange the sums we get
\[ \sum_{A\in \mathcal{N}_q(G)} \mu_q(A,G) |A|^s =  \sum_{(g_1,\ldots,g_s)\in Hom(F_s/\Gamma^q_s,G)} \sum_{\langle g_1,\ldots,g_s \rangle \subseteq A} \mu_q(A,G) ,\]
where the inner sum runs over all $A\in\mathcal{N}_q(G)$ containing $\langle g_1,\ldots , g_s\rangle$. 
Recall that $\mu_q(G,G)=1$, thus 
\[ \sum_{A\in \mathcal{N}_q(G)} \mu_q(A,G) |A|^s =  \sum_{(g_1,\ldots,g_s)\in Hom(F_s/\Gamma^q_s,G)} \left( 
-1+\sum_{\langle g_1,\ldots,g_s \rangle \subseteq A\subseteq G} \mu_q(A,G) \right). \]
By the previous Lemma the inner sum is zero. Therefore
\[ \sum_{A\in \mathcal{N}_q(G)} \mu_q(A,G) |A|^s = \sum_{(g_1,\ldots,g_s)\in Hom(F_s/\Gamma^q_s,G)} (-1) = -|Hom(F_s/\Gamma^q_s,G)|. \]
The result follows. 
\end{proof}

\begin{rem} The previous result allows to regard $P_q(G,s)$ as a function defined over the complex numbers. 
The reciprocal function $P_q(G,s)$ may be called the probabilistic zeta function of $G$ of class $q$.
\end{rem}

\begin{exm}\label{example-series} If $G$ is a TC-group with maximal abelian subgroups $M_1,\ldots,M_N$, then we have
\[ P_2(G,s)=\frac{1-N}{|G:Z(G)|^s}+\sum_{i=1}^{N}\frac{1}{|G:M_i|^s} .\]
We have the following computations:
\begin{enumerate}
\item \[ P_2(\Sigma_3,s)=\frac{1}{2^s} +\frac{3}{3^s} -\frac{3}{6^s}.\]
\item If $D_{2n}$ is the dihedral group of order $2n$ and $n=2k\geq 4$, then
\[ P_2(D_{2n},s)=\frac{1}{2^s}+\frac{k}{k^s}-\frac{k}{n^s}. \]
\item \[ P_2(A_4,s)=\frac{1}{3^s} +\frac{4}{4^s} - \frac{4}{12^s}. \]
\item \[ P_2(A_5,s)=\frac{6}{12^s}+\frac{5}{15^s}+\frac{10}{20^s}-\frac{20}{60^s} .\]
Whereas from~\cite{Boston} we have
\[ P(A_5,s)=1-\frac{5}{5^s}-\frac{6}{6^s}-\frac{10}{10^s}+\frac{20}{20^s}+\frac{60}{30^s}-\frac{60}{60^s}. \]
\item \[ P_2(\Sigma_4,s)=\frac{7}{6^s} +\frac{4}{8^s}-\frac{6}{12^s}-\frac{4}{24^s} .\]
\item \[P_2(PSL(2,7),s)=\frac{8}{24^s}+\frac{35}{42^s}+\frac{28}{56^s}-\frac{42}{84^s}-\frac{28}{168^s}. \]
\item \[ P_3(\Sigma_4,s)=\frac{3}{3^s}-\frac{2}{6^s}+\frac{4}{8^s}-\frac{4}{24^s}.\]

\end{enumerate}
\end{exm}

\begin{prop} If $q\geq 2$, then $\chi(E(q,G))=P_q(G,-1)$.
\end{prop}

\begin{cor} Let $m_q(G)$ be the greatest common divisor of the indices of the maximal subgroups in $\mathcal{N}_q(G)$.
Then $m_q(G)$ divides $\chi(E(q,G))$. 
\end{cor}

\begin{exm} If $G$ is a non-abelian $p$-group, then $\chi(E(2,G))$ is divisible by $p^\alpha$ where $p^\alpha$ is
the index of a maximal abelian subgroup of $G$.
\end{exm}

Lets consider the simple group $PSL(2,p)$ where $p$ is a prime number greater than 3. The order of $PSL(2,p)$ is 
$2pqr$ where $q=\frac{1}{2}(p-1)$ and $r=\frac{1}{2}(p+1)$. The numbers $p,q,r$ are relatively prime. The maximal
abelian subgroups of $PSL(2,p)$ are $\Z/p,\Z/q,\Z/r$ and $D_4$, where $D_{2n}$ is the dihedral group of order $2n$.


\begin{prop} Let $p$ be a prime number greater than 3. Then
\[ P_2(PSL(2,p),s) = \frac{p+1}{(2qr)^s} +  \frac{pr}{(2pr)^s} + \frac{pq}{(2pq)^s}  + \frac{\frac{1}{6}pqr}{(\frac{1}{2}pqr)^s } 
- \frac{ \frac{1}{2}pqr}{(pqr)^s}  - \frac{ p+pr+pq-\frac{pqr}{3}}{ (2pqr)^s }    .\]

\end{prop}

\medskip
\begin{rem} Notice that the isomorphism $Hom(A,B\times C)\cong Hom(A,B)\times Hom(A,C)$ induces a homeomorphism 
$E(q,G\times H)\cong E(q,G)\times E(q,H)$, and produces the identity
\[ P_q(G\times H,s)=P_q(G,s)P_q(H,s) .\]
This tells us that 
\[ \chi(E(q,G\times H))=\chi(E(q,G))\chi(E(q,H)), \] which is
is simpler than the formula obtained in~\cite{Brown} for the poset of proper cosets. What makes the ordinary coset poset
more complex with respect to products is that we do not have, in general, an isomorphism between $BN_n(G\times H)$ and
$BN_n(G) \times BN_n(H)$.
\end{rem}

This latter formula allows to compute the Euler characteristic of $\mathcal{C}_q(G)$ for any nilpotent group $G$ in terms of that of
its Sylow supgroups.


\begin{exm} Let $G$ be the central product of $Q_8$ with itself. $G$ is an extraspecial group of order 32. The maximal abelian subgroups of $G$
have order 8  and the center of $G$ is cyclic of order 2. It follows that $E(2,G)$ has the homotopy type of a 2-dimensional complex.
We can see that
\[ \mu_2(H,G)=\left\{ \begin{array}{rcl}
-1 & if & |H|=8 \\
2 & if & |H|=4 \\ 
-16 & if & |H|=2 \\ \end{array}\right. \]
and find that $\chi(E(2,G))= 76$.
We also have 
\[ P_2(G,s)= \frac{15}{4^s}-\frac{30}{8^s} + \frac{16}{16^s} \]

We can also compute the homology of $E(2,G)$ by using MAGMA:

\[ H_*(E(2,G);\Z) = \left\{ \begin{array}{ccc}
\Z & if & *=0 \\
\Z/2 & if & *=1 \\
\Z^{75} & if & *=2 \\
0 & if & *\geq 3 .
\end{array}\right.
 \] 
Note that this is the same as the homology of $\R P^2\vee \bigvee^{75} S^2$.
\end{exm}

\begin{rem} This latter example shows that $E(q,G)$ does not have to have the homotopy type of a bouquet of spheres when $G$ is solvable. 
In fact, if it was a bouquet of spheres of dimension $> 1$, then $\pi_1(B(2,G))=G$.
\end{rem}


\begin{rem} Note that $P_q(G,1)=1$ for all $G$. Define 
\[ R_q(G,S)=1-P_q(G,s)=\sum_{\mathcal{N}_q(G)\cup\{ G\} }\frac{\mu_q(H,G)}{|G:H|^s} \]
Thus $-R_q(G,-1)=\tilde\chi(\mathcal{C}_qG))$,  the reduced Euler characteristic of $\mathcal{C}_q(G)$. This function $R_q(G,S)$ has a probabilistic interpretation:
$R_q(G,s)$ is the probability that a randomly chosen ordered $s$-tuple does not generate a nilpotent subgroup of class $< q$. 
The inclusions $Hom(F_n/\Gamma^q_n,G)\subseteq Hom(F_n/\Gamma^{q+1}_n,G)$ show that for all integers $s\geq 1$ we have
\[ R_2(G,s)\geq \cdots\geq R_q(G,s)\geq R_{q+1}(G,s)\geq\cdots \]
\end{rem}

Likewise, the function $P_q(G,s)-P_{q-1}(G,s)$ is the probability that a randomly chosen $s$-tuple generates a nilpotent subgroup of
class $q-1$. 

\begin{rem} If $G$ is finite, then the series $P_q(G,s)$ stabilize, that is, there is some $q_G$ such that 
$P_{q_G}(G,s)=P_q(G,s)$ for all $q\geq q_G$. The integer $q_G$ is the smallest integer so that $\Gamma^{q_G}(H)=1$ for all
proper, nilpotent subgroups of $G$.
\end{rem}

\subsection{Factorizations}
The ring of finite Dirichlet series with integer coefficients
\[ \mathcal{R} = \left\{ \sum_{n\geq 0}  \frac{a_n}{n^s} \colon a_n\in\Z,\ a_n\neq 0 \mbox{ for finitely many $n$'s} \right\} \] 
is a unique factorization domain as it is a polynomial ring over $\Z$ in $\frac{1}{2^s}, \frac{1}{3^s}, \frac{1}{5^s},\ldots$. 
It is then natural to study the factorization of $P_q(G,s)$ and $R_q(G,s)$ in $\mathcal{R}$.

\begin{exm} 
\begin{enumerate}
\item Recall that $D_{2n}$ is nilpotent if and only if $n$ is a power of $2$. Suppose that $n=2^{r}$ and $r>1$, so $D_{2n}$ has class $r$. 
Then
\[ R_{r-k+1}(D_{2n},s)=1- \frac{1}{2^s}- \frac{2^k}{2^{ks}} + \frac{2^k}{(2^{k+1})^s} = -(1-\frac{1}{2^s}) \prod_{d|k} \Phi_d \left( \frac{2}{2^s} \right) ,\]
where $\Phi_d$ is the $d^{th}$ cyclotomic polynomial, and $1\leq r-k< r$.
Moreover, 
\[ E(r-k+1,D_{2n})\simeq \bigvee_{2^{2k}-1} S^1 .\] 
for $1\leq r-k< r$.

\item If $G$ is the central product of $Q_8$ with itself, then
\[ R_2(G,s)=(1-\frac{1}{2^s})(1-\frac{2}{2^s})(1+\frac{3}{2^s}-\frac{8}{4^s}) .\]

\item If $q=2^n$ and $n\geq 2$, then
\[ R_2(SL(2,q),s)= 1+\frac{q(q+1)}{[q(q-1)]^s}  -\frac{q+1}{[q^2-1]^s} -\frac{\frac{1}{2}q(q-1)}{[q(q-1)]^s}-\frac{\frac{1}{2}q(q+1)}{[q(q+1)]^s} \]
is irreducible.

\item If $M_{11}$ is the Mathieu group of degree 11, then
\[R_2(M_{11},s)=1-\frac{144 }{ 720^s}-\frac{ 55}{ 880^s}-\frac{ 495}{ 990^s}-\frac{ 660}{1320^s} - \frac{ 396}{1584^s}+\frac{330 }{1980^s}
+\frac{660 }{ 2640^s}+\frac{1980 }{ 3960^s}-\frac{561 }{ 7920^s} \]
is irreducible.
\end{enumerate}
\end{exm}

\begin{conjecture} If $G$ is a finite simple group, then $R_2(G,s)$ is irreducible in $\mathcal{R}$.
\end{conjecture}

Note that if $G$ is nilpotent then $P_q(G,s)$ is reducible for all finite $q\geq 2$. More precisely, if the Sylow subgroups of $G$ are 
$P_1,\ldots,P_n$, then $P_q(G,s)=P_q(P_1,s)\cdots P_q(P_n,s)$.

\subsection{Lifting nilpotent tuples} 

Let $H\lhd G$. Recall that the ordinary probabilistic zeta function satisfies $P(G,s)=P(G/N,s)P(G,N,s)$, 
where $P(G,N,s)$ is the probability that a random lift of a generating $s$-tuple is $G/N$ to an $s$-tuple in $G$ generates $G$ (see~\cite{Brown}).
Note that this is a conditonal probability formula. We will say that a $s$-tuple is $q$-nilpotent if the subgroup that it generates is nilpotent of class $<q$.
We will compute the probability that a random lift of a $q$-nilpotent $s$-tuple in $G/N$ is a $q$-nilpotent $s$-tuple in $G$. This latter will be denoted
by $P_q(G,N,s)$.


Let $\phi \in Hom(F_n/\Gamma^q_n, G/N)$ and $K\subseteq G$. Define $\alpha_q(K, \phi)$ as the number of lifts of $\phi$ to $q$-tuples in $K$, and 
$\beta_q(K,\phi)$ as the number of such lifts that generate $K$. Thus
\[ \alpha_q(K,\phi) = \sum_{H\in\mathcal{L}_q(K)} \beta_q(H,\phi). \]
and so
\[ \beta_q(K,\phi) = \sum_{H\in\mathcal{L}_q(K)} \mu_q(H,K)\alpha_q(H,\phi) .\]
Note that $\beta_q(G)=0$ (if $G$ is not nilpotent of class less that $q$). Moreover,   $\alpha_q(H,\phi)$ is equal to $|H\cap N|^s$ if $\phi$ is in $HN/N$, or
zero otherwise. So
\[ \alpha_q(G,\phi) = -   \sum_{\substack{{H\in\mathcal{N}_q(K)} \\  \phi \in HN/N }} \mu_q(H,G) |H\cap N|^s . \]
Note that $\frac{ \alpha_q(G,\phi) }{|N|^s}$ is the probability that a randomly chosen lift of $\phi$ is a $q$-tuple in $G$. This latter depends on $\phi$, in general,
unless $q=\infty$.

Therefore, the total number of lifts of $q$-tuples in $G/N$ to $q$-tuples in $G$ is given by
\begin{align*}
 \Phi_q(G,N,s) &=  \sum_{\phi\in Hom(F_n/\Gamma^q_n,G/N)} \alpha_q(G,\phi) \\
  &= - \sum_{\phi\in Hom(F_n/\Gamma^q_n,G/N)} \sum_{H\in\mathcal{N}_q(K)} \mu_q(H,G) |H\cap N |^s \\
  &= - \sum_{H\in\mathcal{N}_q(K)}  \sum_{\phi\in Hom(F_n/\Gamma^q_n,HN/N)} \mu_q(H,G) |H\cap N |^s \\
  &=  - \sum_{H\in\mathcal{N}_q(K)} \mu_q(H,G) |H|^s .
 \end{align*}  

Thus
\[ P_q(G,N,s) = \frac{ \Phi_q(G,N,s)}{ |N|^s |Hom(F_n/\Gamma^q_n,G/N)|} = \frac{\Phi_q(G,N,s)/|G|^s}{ P_q(G/N,s)} = \frac{P_q(G,s)}{ P_q(G/N,s)} .\]




Unlike the case of the ordinary probabilistic function of $G$, it is not necessarily
true that $P(G,N,s)$ will be a Dirichlet series in $\mathcal{R}$. This latter can be seen by looking at the values of $P_q(G,s)$ and $P_q(G/N,s)$ at $s=-1$ 
(see, for instance, Example~\ref{example-series}). Thus $P_q(G,N,s)$ is rational.

\subsection{Localizations}
We can also consider localizations of the function $P_q(G,s)$ at different primes. There are at least three ``local'' versions of the probabilistic
functions.

\begin{itemize} 

\item Let $p$ be a prime number and let
$\mathcal{N}_q(G)_p=\{ H\in \mathcal{N}_q(G)| |G|_p=|H|_p  \}$. Define
\[P_q(G,s)_p = -\sum_{H\in \mathcal{N}_q(G)_p}\frac{\mu_q(H,G)}{|G:H|^s} .\]
This latter series is the image of $P_q(G,s)$ in the localization of $\mathcal{R}$ at $\frac{1}{p^s}$.
Thus $R_q(G,s)_p=1-P_q(G,s)_p$. 

\item We can twist again Hall's question and this time ask the following: given a fixed prime $p$, what is the probability that a randomly chosen $s$-tuple of
$G$ generates an abelian $p$-subgroup? This is
\[ \frac{Hom(\Z_p,G)}{|G|^s} \]
where $\Z_p$ stands for the $p$-adic integers. This probability can be described by using the poset of abelian $p$-subgroups of $G$.
Note that if the maximal abelian $p$-subgroups of $G$ are maximal abelian then this probability function coincides with $P_2(G,s)_p$ (for example,
take $G=SL(2,\F_{2^n})$). 

\item Let $\mathcal{N}_q(G)_{\{ p\} }= \{ H\in \mathcal{N}_q(G) \colon |G:H|=p^k, \mbox{for some } k\geq 1\}$. We define
\[P_q(G,s)_{\{ p\} } = -\sum_{H\in \mathcal{N}_q(G)_{\{ p\} }}\frac{\mu_q(H,G)}{|G:H|^s} .\]
Thus 
\[P_q(G,s)_{\{ p\} } = -\sum_{k\geq 1} \frac{a_k}{p^{ks}} ,\]
where \[ a_k=\sum_{ \substack{ {H\in \mathcal{N}_q(G)} \\ |G:H|=p^k   } }\mu_q(H,G). \]

\end{itemize}



\section{Commuting tuples in the symmetric group}

Note that $|Hom(\Z^2,\Sigma_n)| = n!p(n)$, where $p(n)$ is the number of partitions of $n$. Thus
\[ P_2(\Sigma_n,2)=\frac{p(n)}{n!} .\]
Computing the number of commuting $s$-tuples in $\Sigma_n$ is more elaborate than the previous observation as we will see.
Let $\mathcal{P}_n$ be the set of all partitions of $n$. A partition of $n$ will be denoted by $1^{a_1}2^{a_2}\cdots n^{a_n}$, where 
$a_k$ is the number of $k$-cycles in the partition. To compute $|Hom(\Z^s,\Sigma_n)|$ one has to note that an element in 
$Hom(\Z^s,\Sigma_n)$ can be identified with a $\Z^s$-set of degree $n$. A $\Z^s$-set of degree $n$ decomposes into transitive $\Z^s$-sets of
degree $a_i$ so that $1^{a_1}2^{a_2}\cdots n^{a_n}$ is a partition of $n$. A transitive $\Z^s$-set of degree $r$ can be 
identified with a subgroup of $\Z^s$ of index $r$.  Let $j_r(\Z^s)$  be the number of subgroups of $\Z^s$ of index $r$. The number
$j_r(\Z^s)$ turns out to be finite for all $s$ and $r$, and can be computed by using the following two formulae (see~\cite{Solomon}, 
and~\cite{LS}, Section 15.~2):
\[ j_r(\Z^s)= \sum_{r_1\cdots r_s=r} r_2r_3^2\cdots r_s^{s-1}, \]
and
\[ j_r(\Z^s)= \sum_{d|r} j_d(\Z^{s-1}) .\]
Thus we have the following result.
\begin{prop} The number of commuting $s$-tuples in $\Sigma_n$ is given by
\[ |Hom(\Z^s,\Sigma_n)| =\sum_{ 1^{a_1}2^{a_2}\cdots n^{a_n} \in \mathcal{P}_n} \frac{n!}{\prod_{i=1}^{n} i^{a_i} a_i! }
\prod_{i=1}^{n} j_i(\Z^s)^{a_i} .\]
\end{prop}

\section{Conjugacy classes, lower bounds, and upper bounds.}

Another important invariant of a group is the number of conjugacy classes of ordered commuting $n$-tuples.
If $n\geq 1$, then we denote  by $k_n(G)$ the number of conjugacy classes of commuting $n$-tuples in $G$. 
In this section we will use the letter $n$ instead of $s$ to stress the fact that all the results concerning $P_2$ are
for non-negative integers. We begin with the following useful lemma.

\begin{lem}
Let $K$ be any discrete group. Then 
\[|Hom(K\times\Z,G)|/|G| = |Hom(K,G)/G| .\]
\end{lem}
\begin{proof}
Note that 
\[ Hom(K\times\Z,G) = \bigsqcup_{\phi\in Hom(K,G)} C_G(\phi). \]
So
\begin{align*} \frac{1}{|G|}| Hom(K\times\Z,G)| &=\frac{1}{|G|} \sum_{\phi\in Hom(K,G)}| C_G(\phi)| 
 =  \sum_{\phi\in Hom(K,G)} \frac{1}{|G:C_G(\phi)|}  \\
 &=  \sum_{(\phi)\in Hom(K,G)/G} 1 =  | Hom(K,G)/G| . \end{align*}
\end{proof}

The previous result shows that 
 \[  k_n(G)= |Hom(\Z^n,G)/G|=|Hom(\Z^{n+1},G)|/|G| =P_2(G,n+1)|G|^n .\]
 This latter in turn shows that
 \[ k_n(G)=-\frac{1}{|G|}\sum_{H\in \mathcal{N}_2(G)}\mu_2(H,G)|H|^{n+1}. \]


\begin{prop} Let $H$ be a subgroup of $G$. Then 
\[ |G:H|^{-2n} P_2(H,n)\leq  P_2(G,n)\leq P_2(H,n),\] 
and therefore 
\[|G:H|^{-1}k_n(H)\leq  k_n(G)\leq |G:H|^n  k_n(H) .\]
\end{prop}
\begin{proof}
We proceed by induction on $n$ to prove all the inequalities. The case $n=1$ is well-known (see~\cite{GR}).  
Suppose that $n>1$. For the  right-hand side inequalities we have

\[ |Hom(\Z^n,G)| = \sum_{g\in G} |Hom(\Z^{n-1},C_G(g))|   \leq    \sum_{g\in G} |C_G(g):C_H(g)|^{n-1} |Hom(\Z^{n-1},C_H(g))| \]

Since $|C_G(g):C_H(g)|\leq |G:H|$ it follows that 

\[ |Hom(\Z^n,G)| \leq \sum_{g\in G} |G:H|^{n-1} |Hom(\Z^{n-1},C_H(g))| \]

On the other hand,
 \begin{align*}
\sum_{g\in G}  |Hom(\Z^{n-1},C_H(g))| & = \sum_{ \phi\in Hom(\Z^{n-1},H)}  |C_G( \phi)| \\
 & = \sum_{ [\phi ]\in Hom(\Z^{n-1},H)/H}  |C_G(\phi)| |H: C_H(\phi)| \\
 & = \sum_{[\phi]\in Hom(\Z^{n-1},H)/H}  |H| |C_G(\phi) :C_H(\phi)| \\
 & \leq \sum_{[\phi]\in Hom(\Z^{n-1},H)/H}  |H| |G:H| \\ 
 & = |G| k_{n-1}(H) \\
 &= |G:H| |Hom(\Z^n,H)|. 
\end{align*}
The right-hand side inequalities follow.

For the left-hand side inequalities we have
\begin{align*} 
|H|^n |Hom(\Z^n,H) &= \sum_{h\in H} |Hom(\Z^{n-1},H)| \leq |H|^n \sum_{h\in H} |C_G(h):C_H(h)|^{n-1}|Hom(\Z^{n-1},C_G(h)| \\
 &\leq  |H|^n |G:H|^{n-1} \sum_{h\in H} |Hom(\Z^{n-1},C_G(h)| \\
 &\leq |G|^{n-1} |H| \sum_{g\in G} |Hom(\Z^{n-1},C_G(g)| \leq |G|^n |Hom(\Z^n,G)|
\end{align*}
The left-hand side inequalities follow.
\end{proof}

\begin{rem} Since $P_2(G\times H, n)=P_2(G,n)P_2(H,n)$, it follows that $k_n(G\times H)= k_n(G)k_n(H)$.

\end{rem}

\begin{prop} Let $N\lhd G$. Then $P_2(G,n)\leq P_2(G/N,n)P_2(N,n)$, and therefore
\[ k_n(G)\leq k_n(G/N)k_n(N). \]
\end{prop}
\begin{proof}
We proceed by induction on $n$. The result for $n=1$ is well-known. For $n>1$ we have

\begin{align*} |Hom(\Z^n,G)| &= \sum_{g\in G} |Hom(\Z^{n-1},C_G(g))| \\
 &\leq \sum_{g\in G} |Hom(\Z^{n-1},C_G(g)N/N)| |Hom(\Z^{n-1},C_N(g))| \\
 &\leq \sum_{g\in G} |Hom(\Z^{n-1},C_{G/N}(gN)| |Hom(\Z^{n-1},C_N(g))| \\
 &= \sum_{xN \in G/N} \sum_{g\in xN}  |Hom(\Z^{n-1},C_{G/N}(gN)| |Hom(\Z^{n-1},C_N(g))| \\
 &= \sum_{xN \in G/N} |Hom(\Z^{n-1},C_{G/N}(xN)| \sum_{g\in xN}   |Hom(\Z^{n-1},C_N(g))| . \\  
\end{align*} 
On the other hand 
\[ \sum_{g\in xN}   |Hom(\Z^{n-1},C_N(g))| = \sum_{\phi \in Hom(\Z^{n-1},N)}   |C_{xN}(\phi)| . \]
Note that if $y$ is in $C_{xN}(\phi)$, then $C_{xN}(\phi)=C_G(\phi)\cap xN= yC_G(\phi)\cap yN= y C_N(\phi)$. So $C_{xN}(\phi)$ is either
empty or a coset of $C_N(\phi)$. Thus
\begin{align*} \sum_{g\in xN}   |Hom(\Z^{n-1},C_N(g))| &\leq \sum_{\phi \in Hom(\Z^{n-1},N)}   |C_{N}(\phi)|  \\
 &= \sum_{[\phi] \in Hom(\Z^{n-1},N)/N}  |N| \\
 &= |N|k_{n-1}(N) = |Hom(\Z^n,N)|. 
\end{align*}
The result follows.
\end{proof}

The following result provides a slight improvement over that in Lemma 2 of~\cite{GR}.

\begin{prop} Let $x$ be an element of $G-Z(G)$ so that $|C_G(x)|$ is as large as possible, and let $p$ be the smallest prime divisor of the index of $Z(G)$ in 
$G$. Let $c$ denote $|C_G(x)|$. Then for $n\geq 1$,
\[ P_2(G,n+1) \leq \frac{ (p^n+\cdots+p+1)c^n }{p^n |G|^n } .\] 
\end{prop}
\begin{proof}
If $\phi=(x_1,\ldots,x_n)$ is in $Hom(\Z^n,G)$, then $C_G(\phi)\subseteq C_G(x_i)$. If at least one of the components of $\phi$ is not in $Z(G)$, then
$|C_G(\phi)|\leq c$.  Then for $n\geq 1$ we have
\[ |Hom(\Z^n,G)| \geq |Z(G)|^n + (k_n(G) - |Z(G)|^n)\frac{|G|}{c} .\]  
Note that $c/|Z(G)| \geq p$. So
\begin{align*} 
P_2(G, n+1) & \leq \frac{c}{|G|}P_2(G,n) + \left(1-\frac{c}{|G|}\right) \frac{|Z(G)|^n}{|G|^n} \\
 &\leq \frac{c}{|G|}P_2(G,n) + \frac{|Z(G)|^n}{|G|^n} \\
 &\leq \frac{c}{|G|}P_2(G,n) + \frac{c^n }{p^n |G|^n} .
\end{align*} 
An inductive argument applied to the latter inequality completes the proof.
\end{proof}

\begin{prop} Let $p$ be the smallest prime divisor of $|G:Z(G)|$, and let $m$ be the index of $Z(G)$ in $G$. Then
\[ \frac{ mn+m -n}{m^{n+1} } \leq  P_2(G,n+1)\leq \frac{ p^{n+1} + p^n-1}{p^{2n+1}} .\]
\end{prop}
\begin{proof}
If $p$ is the smallest prime divisor of $|G|$, then $|G:C_G(\phi)|\geq p$ for all $\phi \not\in Z(G)^n$. So
\[ |Hom(\Z^n,G)| \geq |Z(G)|^n + p(k_n(G) - |Z(G)|^n) ,\]
and thus
\[ k_n(G) \leq \frac{1}{p} (|Hom(\Z^n,G)| + (p-1)|Z(G)|^n ). \]
Dividing by $|G|^n$ we have
\[ P_2(G,n+1) \leq \frac{1}{p} \left( P_2(G,n) + (p-1)\frac{1}{|G:Z(G)|^n} \right) .\] 
Note that $|G:Z(G)|\geq p^2$, else $G/Z$ is a cyclic group of order $p$. So
\[ P_2(G,n+1) \leq \frac{1}{p} \left( P_2(G,n) + \frac{p-1}{p^{2n}} \right) .\] 
An inductive argument applied to the latter inequality completes the proof for the right hand side inequality.
The left hand side inequality is obtained in a similar way.
\end{proof}

In the last result $p$ is of course at least $2$. The same proof yields the following result proved in~\cite{Lescot}.

\begin{cor} If $G$ is a nonabelian group, then
\[ P_2(G,n+1)\leq \frac {3\cdot 2^n -1}{2^{2n+1} }. \] 
\end{cor}


We close this section by providing a family of congruences for $k_n(G)$. The congruence for $k_1(G)$ was proved in~\cite{Mann}.

\begin{prop} Let $p_1,\ldots, p_l$ be the prime divisors of $|G|$, and let $D_n$ be the greatest common divisor of $p^n_1-1,\ldots,p_l^n-1$.
Then
\[ k_{n-1}(G) \equiv |G|^{n-1} \pmod{ D_n}. \]
\end{prop}

The idea to prove this is as follows: first recall that 
\[ \sum_{H\in \mathcal{L}_2(G)}\mu_2(H,G) =0 .\]
So 
\[ -|G|k_{n-1}(G) = -|G|^n +\sum_{H\in \mathcal{L}_2(G)}\mu_2(H,G)(|H|^n+a),\]
for any number $a$.  Let us restrict to the case $a=-1$. 
Thus, the last three equations show that if $d$ divides $|H|^n -1$ for all $H\subseteq  G$, then $k_{n-1}|G| \equiv |G|^{n} \pmod d .$
If $D_n$ is as defined in the proposition, then one can see that $D_n$ is relatively prime to $|G|$ and that it divides $|H|^n-1$ for all $H\subseteq G$. This 
yields the desired congruence.

\end{document}